\documentclass[12pt]{amsart}
\usepackage{enumerate, amssymb}
\usepackage[colorlinks, bookmarks=true]{hyperref}

\newtheorem{theorem}{Theorem}[section]
\newtheorem{proposition}[theorem]{Proposition}
\newtheorem{lemma}[theorem]{Lemma}

\theoremstyle{definition}

\newtheorem{corollary}[theorem]{Corollary}

\theoremstyle{remark}
\newtheorem{remark}[theorem]{Remark}

\theoremstyle{problem}
\newtheorem{problem}[theorem]{Problem}

\numberwithin{equation}{section}
\def\sqr#1#2{{\,\vcenter{\vbox{\hrule height.#2pt\hbox{\vrule width.#2pt
height#1pt \kern#1pt\vrule width.#2pt}\hrule height.#2pt}}\,}}

\begin{document}

\title[Weak-2-local isometries]{Weak-$2$-local isometries on uniform algebras and Lipschitz algebras}
\author[Li]{Lei Li}
\address[L. Li]{School of Mathematical Sciences and LPMC,
  Nankai University, Tianjin 300071,  China}
\email{leilee@nankai.edu.cn}

\author[A. M. Peralta]{Antonio M. Peralta}
\address[A.M. Peralta]{Departamento de An\'{a}lisis Matem\'{a}tico, Facultad de Ciencias, Universidad de Granada, 18071 Granada, Spain}
\email{aperalta@ugr.es}

\author[L. Wang]{Liguang Wang}
\address[L. Wang]{School of Mathematical Sciences, Qufu Normal University, Qufu 273165, China}
\email{wangliguang0510@163.com}

\author[Y.-S. Wang]{Ya-Shu Wang}
\address[Y.-S. Wang]{Department of Applied Mathematics, National Chung Hsing University, Taichung 402, Taiwan}
\email{yashu@nchu.edu.tw}
\date{\today}
\subjclass[2000]{primary 46B04; 46B20; 46J10; 46E15 secondary: 30H05; 32A38; 46J15; 47B48; 47B38; 47D03}
\keywords{$2$-local isometries; uniform algebras, Lipschitz functions, spherical Gleason-Kahane-Zelazko theorem, spherical Kowalski-S{\l}odkowski theorem}

\maketitle
\begin{abstract}
We establish spherical variants of the Gleason-Kahane-Zelazko and Kowalski-S{\l}odkowski theorems, and we apply them to prove that every weak-2-local isometry between two uniform algebras is a linear map. Among the consequences, we solve a couple of problems posed by O. Hatori, T. Miura, H. Oka and H. Takagi in 2007. 

Another application is given in the setting of weak-2-local isometries between Lipschitz algebras by showing that given two metric spaces $E$ and $F$ such that the set Iso$((\hbox{Lip}(E),\|.\|),(\hbox{Lip}(F),\|.\|))$ is canonical, then every\hyphenation{every} weak-2-local Iso$((\hbox{Lip}(E),\|.\|),(\hbox{Lip}(F),\|.\|))$-map $\Delta$ from $\hbox{Lip}(E)$ to $\hbox{Lip}(F)$ is a linear map, where $\|.\|$ can indistinctly stand for $\|f\|_{_L} := \max\{L(f), \|f\|_{\infty} \}$ or $ \|f\|_{_s} := L(f) + \|f\|_{\infty}.$
\end{abstract}

\section{Introduction}

Let $\hbox{Iso}(X,Y)$ denote the set of all surjective linear isometries between two Banach spaces $X$ and $Y$. Clearly Iso$(X,Y)$ can be regarded as a subset of the space $L(X,Y)$ of all linear maps between Banach spaces $X$ and $Y$. We shall write Iso$(X)$ instead of Iso$(X,X)$. Accordingly to the notation in \cite{NiPe2014,NiPe2015,CaPe2015,CaPe2017} and \cite{Semrl97}, we shall say that a (non-necessarily linear nor continuous) mapping $\Delta : X\to Y$ is a \emph{weak-2-local Iso$(X,Y)$-map} or a \emph{weak-2-local isometry} (respectively, a \emph{2-local Iso$(X,Y)$-map} or a \emph{2-local isometry}) if for each $x,y\in X$ and $\phi \in Y^*$, there exists $T_{x,y,\phi}$ in $\hbox{Iso}(X,Y)$, depending on $x,$ $y$, and $\phi$ (respectively, for each $x,y\in X$, there exists $T_{x,y}$ in $\hbox{Iso}(X,Y)$, depending on $x$ and $y$), satisfying $$\phi \Delta(x) =\phi T_{x,y,\phi}(x), \hbox{ and  } \phi \Delta(y) = \phi T_{x,y,\phi}(y)$$ (respectively, $\Delta(x) = T_{x,y}(x),$  and $\Delta(y) = T_{x,y}(y)$). A Banach space $X$ is said to be (\textit{weak-{\rm)}$2$-iso-reflexive} if every (weak-)$2$-local isometry on $X$ is both linear and surjective.\smallskip

If in the above definition $A$ and $B$ are Banach algebras and the set $\hbox{Iso}(A,B)$ is replaced with the set $\mathcal{G}(A,B)$ of all algebra isomorphisms from $A$ into $B$ we obtain the notion of \emph{(weak-)2-local isomorphism}. For $A=B$ {(weak-)2-local isomorphisms} are called \emph{(weak-)2-local automorphisms}.\smallskip

For $1\leq p<\infty$ and $p\neq 2$, Al-Halees and Fleming \cite{AF09} showed that $\ell^p$ is $2$-iso-reflexive. In the setting of $B(H)$, C$^*$-algebras and JB$^*$-triples there exists a extensive literature on different classes of (weak-)2-local of maps (see, for example, \cite{AyuKuday2012,AyuKuday2014,AyuKudPe2014,BuFerGarPe2015RACSAM,BuFerGarPe2015JMAA,CaPe2015,CaPe2017,Fos2012,Fos2014,G01,JorPe, KOPR2014,NiPe2014,NiPe2015} and \cite{Semrl97}).\smallskip

There are two main questions treated in recent times, the first asks whether every (weak-)2-local isometry between certain Banach spaces $X$ and $Y$ is a linear mapping. If the answer to the first one is affirmative, is every (weak-)2-local isometry between $X$ and $Y$ a surjective isometry?\smallskip

Let $E$ and $F$ denote two metric spaces. A function $f: E\to {F}$ is called \emph{Lipschitz} if its Lipschitz number $$L(f) :=\sup\left\{ \frac{d_{_F} (f(x),f(y))}{d_{_E}(x,y)} : x,y\in E, \ x\neq y \right\}$$ is finite. When $F=Y$ is a Banach space, the symbol Lip$(E,Y)$ will denote the space of all bounded Lipschitz functions from $E$ into $Y$. The space Lip$(E,Y)$ is a Banach space with respect to the following equivalent and complete norms $$\|f\|_{_L} := \max\{L(f), \|f\|_{\infty} \},\hbox{ and } \|f\|_{_s} := L(f) + \|f\|_{\infty}.$$
Throughout this note $\mathbb{F}$ will either stand for $\mathbb{R}$ or $\mathbb{C}$, and we shall write Lip$(E)$ for the space Lip$(E,\mathbb{C})$. It is known that, for every metric space $E$, $(\hbox{Lip}(E),\|.\|_{_L})$ is a unital commutative complex Banach algebra with respect to the pointwise multiplication (see \cite[Propositions 1.5.3 and 1.6.2]{Weav1999}). We refer to the monograph \cite{Weav1999} for the basic notions on Lipschitz algebras.\smallskip

A. Jim{\'e}nez-Vargas and M. Villegas-Vallecillos develop in \cite{JV11} a detailed study on 2-local isometries on Lip$(E)$.  In \cite[Theorem 2.1]{JV11} they prove that for each bounded metric space $E$ such that Iso$((\hbox{Lip}(E),\|.\|_{_L}))$ is canonical, then for every 2-local isometry $\Delta : (\hbox{Lip}(E),\|.\|_{_L})\to (\hbox{Lip}(E),\|.\|_{_L})$ there exists a subset $E_0$ of $E$, a unimodular scalar $\tau$ and a bijective Lipschitz map $\varphi: E_0 \to  E$ such that $$\Delta (f)|_{E_0} \equiv \tau ( f \circ \varphi ),$$ for all $f \in Lip(E)$, and the same statement holds when $\|.\|_{_L}$ is replaced with $\|.\|_{s}$ (see Section \ref{sec:2} for the concrete definition of canonical surjective isometry). However, we do not have a complete control of $\Delta(f)$ outside $E_0$ to conclude that $\Delta$ is linear. If, under the above hypothesis, we additionally assume that $E$ is separable, then every 2-local isometry on $(\hbox{Lip}(E),\|.\|_{_L}),$ or on $(\hbox{Lip}(E),\|.\|_{s}),$ is a surjective linear isometry (see \cite[Theorem 3.3]{JV11}). In this note we shall extend the study to weak-2-local isometries between Lip$(E)$ algebras.\smallskip

We are also interested in some other classes of function algebras. Suppose $K$ is a compact Hausdorff space. The norm closed subalgebras of $C(K)$ containing the constant functions and separating the points of $K$ (i.e., for every $t\neq s$ in $K$ there exists $f \in A$ such that $f (t) \neq f (s)$) are called \emph{uniform algebras}. An abstract characterization of uniform algebras can be deduced from the Gelfand theory and the Gelfand-Beurling formula for the spectral radius, namely, if $A$ is a unital commutative complex Banach algebra such that $\|a^{2}\|=\|a\|^{2}$ for all $a$ in $A$, then there is a compact Hausdorff space $K$ such that A is isomorphic as a Banach algebra to a uniform subalgebra of $C(K)$.\smallskip

Useful examples of uniform algebras include certain spaces given by holomorphic properties. Suppose $K$ is a compact subset of $\mathbb{C}^n$, the algebra $A(K)$ of all complex valued continuous functions on $K$ which are holomorphic on the interior of $K$ is an example of uniform algebra. When $K=\mathbb{D}$ is the closed unit ball of $\mathbb{C}$, $A(\mathbb{D})$ is precisely the disc algebra. On the other hand, it is known that Lip$(K)$ is norm dense in $(C(K),\|.\|_{\infty})$ (\cite[Exercise in page 23]{Weav1995}), however there exist continuous functions which are not Lipschitz. Combining this fact with the Stone-Weierstrass theorem, we deduce that Lip$(K)$ is not, in general, a uniform algebra. The reader if referred to the monograph \cite{Gamelin2005} for additional background on uniform algebras.\smallskip

In \cite{HaMiOkTak07} O. Hatori, T. Miura, H. Oka and H. Takagi studied 2-local isometries and 2-local automorphisms between uniform algebras. These authors established some partial answers. As an application of the  Kowalski-S{\l}odkowski theorem (see Theorem \ref{t KS}) they first proved that for each uniform algebra $A$, every 2-local automorphism $T$ on $A$ is an isometrical isomorphism from $A$ onto $T(A)$. Furthermore, if the group of all automorphisms on $A$ is algebraically reflexive (i.e., if every local automorphism on $A$ is an automorphism), then every 2-local automorphism is an automorphism (see \cite[Theorem 2.2]{HaMiOkTak07}). They also showed the existence of non-surjective 2-local automorphisms on $C(K)$ spaces (see \cite[Theorem 2.3]{HaMiOkTak07}). For a compact subset $K\subseteq \mathbb{C}$ such that int$(K)$ is connected
and $\overline{\hbox{int}(K)} = K$, Hatori, Miura, Oka and Takagi combined their results on local automorphism and local isometries on $A(K)$ with a previous contribution by F. Cabello and L. Molnar (see \cite{CaMol2002}) to prove that every local isometry (respectively, every local automorphism) on $A(K)$ is a surjective isometry (respectively, is an automorphism) (see \cite[Corollary 3.5]{HaMiOkTak07}). Under certain restrictions on a set $K\subset \mathbb{C}$ or $K\subset \mathbb{C}^2$, the same authors establish that every 2-local isometry (respectively, every 2-local automorphism) on $A(K)$ is a surjective linear isometry (respectively, an automorphism) \cite[Theorems 3.6, 3.6 and 3.8]{HaMiOkTak07}. In the same paper, these authors posed the following problems

\begin{problem}\label{problem linearity of 2-local isometries for uniform}\cite[Problem 3.12]{HaMiOkTak07} Is every 2-local isometry on a uniform algebra linear?
\end{problem}

\begin{problem}\label{problem linearity of n-local isometries for uniform}\cite[Problem 3.13]{HaMiOkTak07} Is a 2-local isometry (respectively automorphism) on a uniform algebra a 3-local isometry (respectively automorphism)? In general, is an $n$-local isometry an $(n + 1)$-local
isometry?
\end{problem}

We recall that given a natural $n$, a map $\Delta$ on a Banach algebra $B$ is called an \emph{$n$-local isometry} (respectively, an $n$-local
automorphism) if for every $n$-tuple $(a_1, a_2, \ldots , a_n)$ in $B$ there exists a surjective linear isometry (respectively, an automorphism) $S$ on $B$, depending on the elements in the tuple,  such that $\Delta (a_j) = S(a_j)$ for every $1 \leq j \leq n$.\smallskip

In this paper we give a complete positive answer to the above problems by proving that every weak-2-local isometry between uniform algebras is linear (see Theorem \ref{t 2-local uniform algebras}). This conclusion is actually stronger than what it was posed by Hatori, Miura, Oka and Takagi. In order to prove this result we first establish appropriate spherical variants of the Gleason-Kahane-Zelazko and Kowalski-S{\l}odkowski theorems (see Propositions \ref{p GKZ sphere} and \ref{p KS sphere}). The spherical version of the Gleason-Kahane-Zelazko theorem is employed in section \ref{sec:2} to describe weak-local isometries on uniform algebras and on Lipschitz spaces.\smallskip

Henceforth, let the symbol $\mathbb{T}$ stand for the unit sphere of $\mathbb{C}$.
Throughout this note, the symbol $\sigma(a)$ will denote the spectrum of an element $a$ in a complex Banach algebra $A$.
The key result in this note is an spherical variant of the Kowalski-S{\l}odkowski which assures that for an arbitrary complex unital Banach algebra $A$, a mapping  $\Delta : A\to \mathbb{C}$ which is 1-homogeneous (i.e. $\Delta (\alpha x) = \alpha \Delta(x)$ for all $\alpha\in \mathbb{C}$ and $x\in A$) and satisfies that $$\Delta (x) - \Delta (y) \in \mathbb{T} \ \sigma (x-y), \hbox{ for every $x,y\in A$},$$ must be linear. Furthermore, under these hypothesis, there exists $\lambda_0$ in $\mathbb{T}$ such that $\lambda_0 \Delta$ is multiplicative (see Proposition \ref{p KS sphere}).\smallskip

We shall also apply the spherical variant of the Kowalski-S{\l}odkowski theorem in the study of weak-2-local isometries between Lipschitz\hyphenation{Lipschitz} algebras. The concrete result reads as follows: Let $E$ and $F$ be metric spaces such that the set Iso$((\hbox{Lip}(E),\|.\|_{_L}),(\hbox{Lip}(F),\|.\|_{_L}))$ is canonical. Then every\hyphenation{every} weak-2-local Iso$((\hbox{Lip}(E),\|.\|_{_L}),(\hbox{Lip}(F),\|.\|_{_L}))$-map $\Delta$ from $\hbox{Lip}(E)$ to $\hbox{Lip}(F)$ is a linear map. Furthermore, the same conclusion holds when the norm $\|.\|_{_L}$ is replaced with the norm $\|.\|_{s}$  (see Theorem \ref{t 2local isometries between Lip spaces}). Among the consequences of this result we establish that for each compact metric space $K$ such that Iso$(\hbox{Lip}(K), \|.\|_{_s})$ is canonical, every 2-local isometry $\Delta : (\hbox{Lip}(K), \|.\|_{_s})\to (\hbox{Lip}(K), \|.\|_{_s})$ is a surjective linear isometry (see Corollary \ref{c 2-local isometries on LipX compact sum norm}).

\section{A spherical variant of the Gleason-Kahane-Zelazko theorem}\label{sec:2}

In this section we shall try to extend the study on 2-local isometries developed by Jim{\'e}nez-Vargas and Villegas-Vallecillos in \cite{JV11}.\smallskip

To understand the whole picture it is worth to recall the notions of local and weak-local maps. Following \cite{EssaPeRa16,CaPe2015}, let $\mathcal{S}$ be a subset of the space $L(X,Y)$ of all linear maps between Banach spaces $X$ and $Y$. A linear mapping $\Delta : X\to Y$ is said to be a \emph{local $\mathcal{S}$ map} (respectively, a \emph{weak-local $\mathcal{S}$-map}) if for each $x\in X$ (respectively, if for each $x\in X$ and $\phi\in Y^*$), there exists $T_{x}\in \mathcal{S}$, depending on $x$ (respectively, there exists $T_{x,\phi}\in \mathcal{S}$, depending on $x$ and $\phi$),  satisfying $\Delta(x) = T_{x}(x)$ (respectively, $\phi \Delta(x) = \phi T_{x,\phi}(x)$). Local and weak-local maps have been intensively studied by a long list of authors (see, for example, \cite{BurCaPe16,BurFerGarPe2012,BurFerPe2013,CaMol2002,EssaPeRa16,EssaPeRa16b,JiMorVill2010,John01,Kad90,LarSou} and \cite{Mack}).\smallskip

In their study on local isometries between Lipschitz algebras, A. Jim{\'e}nez-Vargas, A. Morales Campoy, and M. Villegas-Vallecillos revealed the connection with the Gleason-Kahane-Zelazko theorem (see \cite[page 199]{JiMorVill2010}); a similar strategy was applied by F. Cabello and L. Moln{\'a}r in \cite{CaMol2002} for local isometries on a uniform algebra. Here, we shall try to determine the connection between 2-local isometries between complex-valued Lipschitz algebras and subtle generalizations of the Gleason-Kahane-Zelazko and Kowalski-S{\l}odkowski theorems. Let us recall the statement of these results. We briefly recall that, accordingly to standard references (see \cite{BonDun}), a complex Banach algebra is an associative algebra $A$ over the complex field which is also a Banach space satisfying $$\|x\,y\|\ \leq \|x\|\,\|y\|,\ \  \forall x, y \in A.$$

\begin{theorem}\label{t GKZ}{\rm(Gleason-Kahane-Zelazko theorem \cite{Gle,KaZe,Ze68})} Let $F:A \to \mathbb{C}$ be a non-zero continuous functional, where $A$ is a complex Banach algebra. Then the following statements are equivalent:\begin{enumerate}[$(a)$]\item $F(a)\in \sigma(a)$, for every $a\in A$;
\item $F$ is multiplicative. $\hfill\Box$
\end{enumerate}
\end{theorem}

We begin with a technical \emph{spherical reformulation} of the Gleason-Kahane-Zelazko theorem.

\begin{proposition}\label{p GKZ sphere} Let $F:A \to \mathbb{C}$ be a continuous functional, where $A$ is a unital complex Banach algebra. Suppose that $F(a)\in \mathbb{T} \ \sigma(a)$, for every $a\in A$. Then the mapping $\overline{F(1)} F$ is multiplicative.
\end{proposition}

\begin{proof} The arguments are very similar to those in \cite{KaZe, Ze68}, we shall insert here a brief revision containing the required changes.\smallskip

We fix $a$ in $A$, and define a mapping $\varphi : \mathbb{C} \to \mathbb{C}$, $\varphi (\lambda) = F(e^{\lambda a})$. Since $F$ is linear and continuous, $\displaystyle \varphi (\lambda) = \sum_{n=0}^{\infty} \frac{F(a^n)}{n!} \lambda^n$ for all $\lambda\in \mathbb{C}.$ Clearly, $\varphi$ is an entire function. By the assumptions $\varphi (\lambda) \in  \mathbb{T} \ \sigma(e^{\lambda a})$, and hence $\varphi (\lambda) \neq 0$ for every $\lambda\in \mathbb{C}$. It is known that every entire function which never vanishes admits a ``logarithm'', that is, there exists $\psi :\mathbb{C}\to \mathbb{C}$ entire such that $\varphi (\lambda) = e^{\psi(\lambda)}$ ($\lambda\in \mathbb{C}$).\smallskip

We also know that the inequality $$ |\varphi (\lambda)| \leq \sum_{n=0}^{\infty} \frac{|F(a^n)|}{n!} |\lambda|^n \leq \|F\| \ e^{|\lambda| \ \|a\|},$$ holds for every complex $\lambda$, and hence $\varphi$ has exponential type bounded by one. It follows from Hadamard's Factorization theorem (see \cite[Theorem 2.7.1]{Boas}) that $\psi (\lambda ) = \alpha \lambda +\beta$ for suitable $\alpha,\beta\in \mathbb{C}$.\smallskip

Another application of the hypothesis gives $e^{\beta} = \varphi (0) = F(1)= \overline{\lambda_0} \in \mathbb{T}$. Therefore, $$ \sum_{n=0}^{\infty} \frac{F(a^n)}{n!} \lambda^n = \varphi (\lambda) = e^{\alpha \lambda +\beta}= e^{\beta} \ \sum_{n=0}^{\infty} \frac{\alpha^n}{n!} \lambda^n,$$ for all $\lambda\in \mathbb{C}$. Consequently, $F(a^n) = e^{\beta} \ \alpha^n $ for all natural $n$, in particular $e^{-\beta} \ F(a) = \alpha$ and $e^{-\beta} \ F(a^2) = \alpha^2 = (e^{-\beta} \ F(a))^2$.\smallskip

We have shown that the mapping $G(a) := \lambda_0 F(a)$ is a linear functional with $G(a^2) = G(a)^2$ for every $a\in A$, that is, $G$ is a Jordan homomorphism (i.e., $G$ preserves products of the form $a\circ b = \frac12 (a b +ba)$). The arguments in \cite{Ze68} can be literally applied to conclude that $G$ is multiplicative (see also \cite[Proposition 16.6]{BonDun}).
\end{proof}

F. Cabello S{\'a}nchez and L. Moln{\'a}r proved in \cite[Theorem 5]{CaMol2002} that for each uniform algebra $A$, every local isometry $T$ on $A$ has the form $$T(f) = \tau \ \psi (f), \ \ (\forall f\in A),$$ where $\tau$ is a unimodular element in $A$, and $\psi : A\to A$ is a unital algebra endomorphism. Actually, the result of Cabello and Moln{\'a}r only requires that for each $f\in A$, the element $T(f)$ lies in the norm closure of Iso$(A) (f) = \{S(f) : S\in \hbox{Iso}(A)\}$. Furthermore, if $A= A(\mathbb{D})$ is the disc algebra, then every local isometry on $A(\mathbb{D})$ is a surjective linear isometry (see \cite[Theorem 6]{CaMol2002}). The spherical version of the Gleason-Kahane-Zelazko theorem established in Proposition \ref{p GKZ sphere} provides a powerful tool to extend the study by Cabello and Moln{\'a}r to the setting of weak-local isometries between uniform algebras.

\begin{theorem}\label{t CabelloMolnar weak local} Let $T : A\to B$ be a weak-local isometry between uniform algebras. Then there exists a unimodular element $u\in B$ and a unital algebra homomorphism  $\psi : A\to B$ such that $$T(f) = u \ \psi (f), \ \ \forall f\in A.$$
\end{theorem}

\begin{proof} Suppose $B$ is a norm-closed subalgebra of some $C(Q)$. Let us fix $s\in Q$, and consider the mapping $\delta_s \circ T: A\to \mathbb{C}$. The hypothesis combined with the deLeeuw-Rudin-Wermer theorem (see \cite[Corollary 2.3.16]{FleJa03}) imply that for each $f\in A$ there exist a unimodular element $u_{s,f}\in B$ and a unital algebra isomorphism  $\psi_{s,f} : A\to B$ such that $$\delta_s \circ T(f)= T(f)(s) = u_{s,f} (s) \ \psi_{s,f} (f) (s)\in \mathbb{T} \ \sigma (f).$$ By Proposition \ref{p GKZ sphere} the mapping $\overline{T(1)(s)} \ (\delta_s \circ T): A\to \mathbb{C}$ is a homomorphism. We also know from the hypothesis and the deLeeuw-Rudin-Wermer theorem that $T(1)$ is a unimodular mapping and $\overline{T(1)} = T(1) \in B$. Since $s$ was arbitrarily chosen, we deduce that $$  \overline{T(1)(s)} \ (\delta_s \circ T) ( fg) = \overline{T(1)(s)} \ (\delta_s \circ T) (f) \ \overline{T(1)(s)} \ (\delta_s \circ T) (g),$$ for all $f,g\in A$, equivalently $\psi =\overline{T(1)} \ T : A\to B$ is a homomorphism and $T= T(1) \ \psi$.
\end{proof}

Back to the setting of Lipschitz algebras, we recall that given a surjective isometry $\varphi : F \to E$ between two metric spaces, and an element $\tau\in S_{\mathbb{F}}$, the mapping $$T_{\tau,\varphi} : \hbox{Lip}(E)\to \hbox{Lip}(F),\ \ T_{\tau,\varphi} (f) (s) = \tau f(\varphi(s)), \ \  (f\in \hbox{Lip}(E)),$$ is an element in Iso$(\hbox{Lip}(E),\hbox{Lip}(F))$. Fortunately or not, there exist elements in Iso$(\hbox{Lip}(E),\hbox{Lip}(F))$ which cannot be written as weighted composition operator via a surjective isometry $\varphi$ and $\tau\in S_{\mathbb{F}}$ as in the previous example (cf. \cite[page 242]{Weav1995} or \cite[\S 2.6]{Weav1999}). The elements in Iso$(\hbox{Lip}(E),\hbox{Lip}(F))$ which can be written as weighted composition operators via a surjective isometry $\varphi$ and $\tau\in S_{\mathbb{F}}$ as above are called \emph{canonical}. We shall say that Iso$(\hbox{Lip}(E),\hbox{Lip}(F))$ is canonical if every element in this set is canonical.\smallskip

Many attempts have been conducted to determine when the set Iso$(\hbox{Lip}(E),\hbox{Lip}(F))$ is canonical (see \cite{Roy68}, \cite{Vasa}, \cite{RaoRoy71}, \cite{Weav1995}). More precisely, Iso$(\hbox{Lip}([0, 1]), \|.\|_s)$ is canonical (see \cite{RaoRoy71}), for a compact and connected metric space $K$ with diameter at most $1$, Iso$(\hbox{Lip}(K), \|.\|_{_L})$ is canonical (cf. \cite{Roy68}, \cite{Vasa}). K. Jarosz and V. Pathak established in \cite[Examples 8 and 3]{JarPathak88} that for every compact metric space $K$, Iso$(\hbox{Lip}(K), \|.\|_{_s})$ is canonical, and under a certain separation property, Iso$(\hbox{Lip}(K), \|.\|_{_L})$ is canonical too. It should be noted here that, as pointed out by N. Weaver in \cite[page 243]{Weav1995}, there is a gap in the arguments applied in \cite[Example 8]{JarPathak88}, and so the statement concerning Iso$(\hbox{Lip}(K), \|.\|_{_s})$ remains as open problem.\label{problems in JaroszPathak}\smallskip

N. Weaver introduced subtle novelties in \cite{Weav1995} by proving that if $E$ and $F$ are complete metric spaces of diameter $\leq 2$ and $1$-connected then Iso$((\hbox{Lip}(E),\|.\|_{_L}),(\hbox{Lip}(F),\|.\|_{_L}))$ is canonical (see \cite[THEOREM D]{Weav1995} or \cite[Theorem 2.6.7]{Weav1999}).\smallskip

In the introduction of \cite{Weav1995} (see also \cite[\S 1.7]{Weav1999}), Weaver illustrates why it is worth to restrict our study to the class of complete metric spaces of diameter $\leq 2$, essentially because given a metric space $(E,d)$, we can consider the metric space $(F_0,d')$ whose underlying set is $E$ and whose distance is given by $d'(x,y) = \min \{d(x,y), 2\}$, whose completion is denoted by $(F,d')$. Then $F$ is a complete metric space of diameter $\leq 2$ and Lip$(E)$ and Lip$(F)$ are isometrically isomorphic as Banach spaces.\smallskip

Let us observe that while  $(\hbox{Lip}(E),\|.\|_{s})$ is a commutative and unital complex Banach algebra, the algebra $(\hbox{Lip}(E),\|.\|_{_L})$ only satisfies the weaker inequality $$ \| f g \|_{_L} \leq 2\ \|f\|_{_L} \ \|g\|_{_L}, \ \ \forall f,g\in \hbox{Lip}(E)$$ (see \cite[\S 4.1]{Weav1999}).\smallskip

In \cite{JiMorVill2010} A. Jim{\'e}nez-Vargas, A. Morales Campoy and M. Villegas-Vallecillos proved the following result:

\begin{theorem}\label{t local isometries between Lip spaces}\cite[Theorem 2.3]{JiMorVill2010} Let $E$ be a compact metric space. Suppose that Iso$(\hbox{Lip}(E), \|.\|_{_s})$ is canonical. Then every local isometry on $(\hbox{Lip}(E),\|.\|_{s})$ is a surjective isometry.$\hfill\Box$
\end{theorem}

The original statement of the above theorem in \cite[Theorem 2.3]{JiMorVill2010} does not include the hypothesis affirming that Iso$(\hbox{Lip}(E), \|.\|_{_s})$ is canonical, the reason being that in \cite[Theorem 2.3]{JiMorVill2010} the authors assume that this hypothesis is automatically true by the result stated in \cite[Example 8]{JarPathak88}. However, we have already commented the difficulties appearing in the arguments (see page \pageref{problems in JaroszPathak}).\smallskip

Next, thanks to the spherical version of the Gleason-Kahane-Zelazko theorem, we extend the study to weak-local isometries between Lipschitz spaces when they are indistinctly equipped with the norm $\|.\|_{s}$ or with the norm $\|.\|_{L}$.

\begin{theorem}\label{t weak local isometries between Lip spaces} Let $E$ and $F$ be metric spaces. Then the following statements hold: \begin{enumerate}[$(a)$]\item Suppose that the set Iso$((\hbox{Lip}(E),\|.\|_{s}),(\hbox{Lip}(F),\|.\|_{s}))$ is canonical. Then every weak-local isometry $T:(\hbox{Lip}(E),\|.\|_{s})\to (\hbox{Lip}(F),\|.\|_{s})$  can be written in the form $$T (f) = \tau \ \psi (f), \ \ \hbox{ for all } f\in \hbox{Lip}(E),$$ where $\tau\in \hbox{Lip}(E)$ is unimodular, and $\psi : \hbox{Lip}(E)\to \hbox{Lip}(F)$ is an algebra homomorphism;
\item Suppose that the set  Iso$((\hbox{Lip}(E),\|.\|_{_L}),(\hbox{Lip}(F),\|.\|_{_L}))$ is canonical. Then every weak-local isometry $T:(\hbox{Lip}(E),\|.\|_{_L})\to (\hbox{Lip}(F),\|.\|_{_L})$  can be written in the form $$T (f) = \tau \ \psi (f), \ \ \hbox{ for all } f\in \hbox{Lip}(E),$$ where $\tau\in \hbox{Lip}(E)$ is unimodular, and $\psi : \hbox{Lip}(E)\to \hbox{Lip}(F)$ is an algebra homomorphism.
\end{enumerate}
\end{theorem}

\begin{proof} We shall give a unified proof for both statements. To simplify the notation let $T:(\hbox{Lip}(E),\|.\|)\to (\hbox{Lip}(F),\|.\|)$ be a weak-local isometry, where $\|.\|$ denotes $\|.\|_{_L}$ or $\|.\|_{s}$. By the Hahn-Banach theorem every weak-local isometry is continuous and contractive.\smallskip

By hypothesis, for each $s\in F$, there exists $\tau_{s,1}\in\mathbb{T}$ and a surjective isometry $\varphi_{s,1} : F \to E$ such that $T(1) (s) = \delta_s(T(1)) = \delta_s(\tau_{s,1} \ 1) = \tau_{s,1}$. Therefore $T(1)(s)\in \mathbb{T}$ for all $s\in F$. Similar arguments show that for each $f\in \hbox{Lip}(E)$, there exists $\tau_{s,f}\in\mathbb{T}$ and a surjective isometry $\varphi_{s,f} : F \to E$ such that $$T(f) (s) = \delta_s(T(f)) 
= \tau_{s,f} \ f(\varphi_{s,f}(s))\in \mathbb{T} \ \sigma_{_{({Lip}(E),\|.\|_{s})}}(f).$$ 

In case $(a)$ we know that $(\hbox{Lip}(E),\|.\|_{s})$ is a commutative unital complex Banach algebra, and $\delta_s \circ  T : \hbox{Lip}(E) \to \mathbb{C}$ is continuous. We can therefore apply Proposition \ref{p GKZ sphere} to conclude that $$\overline{T(1)(s)} (\delta_s \circ  T): (\hbox{Lip}(E),\|.\|_{s}) \to \mathbb{C}$$ is a homomorphism for all $s\in F$. Consequently, the operator $\psi= \overline{T(1)} T: \hbox{Lip}(E)\to \hbox{Lip}(F) $ is a homomorphism and $T (f) = T(1) \ \psi (f),$ for all $f\in \hbox{Lip}(E)$.\smallskip

In case $(b)$, as we have commented above, $(\hbox{Lip}(E),\|.\|_{_L})$ is not strictly a Banach algebra (see comments before Theorem \ref{t local isometries between Lip spaces}). However, the algebras underlying $(\hbox{Lip}(E),\|.\|_{s})$ and $(\hbox{Lip}(E),\|.\|_{_L})$ both are the same. Furthermore, for each $s\in F$ the mapping $\delta_s \circ  T : (\hbox{Lip}(E),\|.\|_{_L}) \to \mathbb{C}$ can be also considered as a continuous functional on $(\hbox{Lip}(E),\|.\|_{s})$ satisfying the hypothesis of Proposition \ref{p GKZ sphere}, we can therefore apply case $(a)$ to conclude the proof.
\end{proof}

We observe that statement $(b)$ in the above Theorem \ref{t weak local isometries between Lip spaces} is a new result even in the setting of local isometries. When the group of surjective isometries is not canonical the spherical Gleason-Kahane-Zelazko theorem can not be applied. For example, let $E=\{p,q\}$ be the compact metric space consisting of two elements with $d(p,q) =1$. We can identify Lip$(E)$ with $\mathbb{C}^2$ with norm $\|(a,b)\|_{_L} = \max\{|a|,|b|,|a-b|\}$. The mapping $T: \hbox{Lip}(E)\to \hbox{Lip}(E)$ defined by $T(a,b) = (a,a-b)$ is a surjective linear isometry which is not canonical (see \cite{Weav1995}). Furthermore, the functional $\delta_q \circ T: \hbox{Lip}(E)\to \mathbb{C}$, $(a,b) \mapsto a-b$ does not satisfy that $\delta_q \circ T (f) \in \mathbb{T}\ f(E) = \mathbb{T} \ \sigma(f)$ for all $f\in \hbox{Lip}(E)$.

\begin{remark}\label{remark arguments valid if Iso is not canonical}
If in Theorem \ref{t weak local isometries between Lip spaces} the hypothesis $$\hbox{Iso$((\hbox{Lip}(E),\|.\|_{s}),(\hbox{Lip}(F),\|.\|_{s}))$ being canonical}$$ (respectively, Iso$((\hbox{Lip}(E),\|.\|_{_L}),(\hbox{Lip}(F),\|.\|_{_L}))$ being canonical) is replaced by the weaker assumption that every element $T$ in the set Iso$((\hbox{Lip}(E),\|.\|_{s}),(\hbox{Lip}(F),\|.\|_{s}))$ (respectively, in the set \linebreak Iso$((\hbox{Lip}(E),\|.\|_{_L}),(\hbox{Lip}(F),\|.\|_{_L}))$) is of the form $$T(f ) (s)  = \tau (s) \ f(\varphi(s)), \ \hbox{ for all } f\in \hbox{Lip}(E), s\in F,$$ where $\tau$ is a unimodular function in Lip$(F)$ and $\varphi: F \to E$ is a surjective isometry, then the conclusion of the quoted Theorem \ref{t weak local isometries between Lip spaces} remains true.
\end{remark}

\section{A spherical variant of the Kowalski-S{\l}odkowski theorem}\label{sec:3}

The Kowalski-S{\l}odkowski theorem is less widely known that the Gleason-Kahane-Zelazko theorem, however it shows that at the cost of requiring a 2-local behavior, the linearity assumption in the Gleason-Kahane-Zelazko theorem can be relaxed. Actually, the Kowalski -S{\l}odkowski theorem was the motivation to consider 2-local automorphisms (see \cite{Semrl97}). The concrete result reads as follows.

\begin{theorem}\label{t KS}{\rm(Kowalski-S{\l}odkowski theorem \cite{KoSlod})} Let $A$ be a complex Banach algebra, and let $\Delta : A\to \mathbb{C}$ be a mapping satisfying $\Delta (0)=0$ and $$\Delta (x) - \Delta (y) \in \sigma (x-y),$$ for every $x,y\in A$. Then $\Delta$ is linear and multiplicative. $\hfill\Box$
\end{theorem}

As in the case of the Gleason-Kahane-Zelazko theorem, we shall establish a \emph{spherical variant} of the Kowalski-S{\l}odkowski theorem. We shall say that a mapping $\Delta: X\to Y$ between complex Banach spaces is 1-\emph{homogeneous} or simply \emph{homogeneous} if $\Delta (\lambda x) = \lambda \Delta(x)$, for every $x\in X$, $\lambda\in \mathbb{C}$.

\begin{proposition}\label{p KS sphere} Let $A$ be a unital complex Banach algebra, and let $\Delta : A\to \mathbb{C}$ be a mapping satisfying the following properties: \begin{enumerate}[$(a)$]\item $\Delta$ is 1-homogeneous;
\item $\Delta (x) - \Delta (y) \in \mathbb{T} \ \sigma (x-y),$ for every $x,y\in A$.
\end{enumerate} Then $\Delta$ is linear, and there exists $\lambda_0\in \mathbb{T}$ such that $\lambda_0 \Delta$ is multiplicative.
\end{proposition}

Let $\mathcal{S}$ be a subset of the space $L(X,Y)$ of all linear maps between two complex Banach spaces. Suppose that $\Delta : X\to Y$ is a weak-2-local $\mathcal{S}$ map in the sense of \cite{CaPe2015} (i.e. for every $x,y\in X$ and $\phi\in Y^*$, there exists $T_{x,y,\phi}\in \mathcal{S}$, depending on $x,y,$ and $\phi$ such that $\phi \Delta(x) = \phi T_{x,y,\phi} (x)$ and $\phi \Delta(y) = \phi T_{x,y,\phi} (y)$). It follows from \cite[Lemma 2.1]{CaPe2015} that every weak-2-local Iso$(X,Y)$ is 1-homogeneous. So, the hypothesis in the above proposition are fully motivated in the study of weak-2-local isometries.\smallskip

The arguments here are somehow inspired by the original ideas due to Kowalski and S{\l}odkowski. Before dealing with the proof of the above proposition we establish a variant of \cite[Lemma 2.1]{KoSlod}. Let us first make an observation. Given a compact Hausdorff space $K$ and $t\in K$, the mapping $\overline{\delta}_t : C(K) \to \mathbb{C}$, $f\mapsto \overline{\delta}_t (f) = \overline{f(t)}$ is conjugate linear, and satisfies the property that $\overline{\delta}_t (f)\in \mathbb{T} \sigma(f)$, for each $f\in C(K)$. So, the just quoted property is not enough to guarantee that a real linear functional is complex linear.

\begin{lemma}\label{l KS sphere} Let $A$ be a unital complex Banach algebra, and let $F : A\to \mathbb{C}$ be a bounded real linear functional satisfying the following property: $$F(x)\in \mathbb{T} \ \sigma(x),\hbox{ for every }x\in A.$$
Then $F$ is complex linear or conjugate linear.
\end{lemma}

\begin{proof} In a first step we assume that $F$ is unital, that is $F(1) =1$. Fix an element $x\in A$. By hypothesis, we have $$e^{i t} F(e^{-i t} x) \in e^{i t} \mathbb{T} \sigma (e^{-i t} x) = \mathbb{T} \ \sigma (x)$$ for each real $t$. As in the proof of \cite[Lemma 2.1]{KoSlod} we can write $$ e^{i t} F(e^{-i t} x) = \frac{F(x) -i F(ix)}{2} + e^{2 i t} \frac{F(x) + i F(i x)}{2},$$ which shows that $e^{i t} F(e^{-i t} x)$ is precisely the circle $S(\frac{F(x) -i F(ix)}{2} ,\rho_x)$ of center $\frac{F(x) -i F(ix)}{2} $ and radius $\rho_x= \left|\frac{F(x) + i F(i x)}{2}\right|$, and consequently this circle is entirely contained in $\mathbb{T} \ \sigma (x)$. It is not hard to check that the functionals $$F_1 (x) := \Re\hbox{e} F(x) + i \Im\hbox{m} (-i F(ix)) = \Re\hbox{e} F(x) - i \Re\hbox{e} F(ix)$$
$$F_2 (x) := \Re\hbox{e} (-i F(i x)) + i \Im\hbox{m} F(x) = \Im\hbox{m} F(ix) + i \Im\hbox{m} F(x)$$ are continuous and satisfy the following properties:
\begin{enumerate}[$(1)$]\item $F_j (i x ) = i F_j (x)$ for all $j=1,2$ and all $x\in A$, and hence $F_1$ and $F_2$ are complex linear;
\item $F_j (x) \in S(\frac{F(x) -i F(ix)}{2} ,\rho_x)\subseteq \mathbb{T} \ \sigma (x),$ for every $x\in A$, $j=1,2$;
\item $\Re\hbox{e} (F(x)) = \Re\hbox{e} (F_1(x))$ and $\Im\hbox{m} (F(x)) = \Im\hbox{m} (F_2(x))$ for all $x\in A$.
\end{enumerate}

Property $(2)$ implies that we can apply Proposition \ref{p GKZ sphere} to conclude that $G_1 = \overline{F_1(1)} F_1$ and $G_2 = \overline{F_2(1)} F_2$ are multiplicative functionals.\smallskip

It also follows from $(2)$ that $F_1(1) \in \mathbb{T}$. Since $F(1) =1$ we have $$ \mathbb{T} \ni F_1(1) =  \Re\hbox{e} F(1) - i \Re\hbox{e} F(i 1) = 1 - i \Re\hbox{e} F(i 1), $$ which assures that $ \Re\hbox{e} F(i 1) =0$ and hence $F_1(1) = 1$. This shows that $F_1 = G_1$ is multiplicative.\smallskip

Similarly, from $(2)$ and $F(1)=1$ we get $$  \mathbb{T} \ni F_2(1) = \Im\hbox{m} F(i 1) + i \Im\hbox{m} F(1) = \Im\hbox{m} F(i 1),$$ witnessing that $F_2(1) =\Im\hbox{m} F(i 1)= \pm 1$.\smallskip

We claim that \begin{equation}\label{eq F1 is pm F2} F_2 = F_2(1) F_1.
\end{equation} Indeed, if $\ker (F_2)\subseteq \ker (F_1)$ then there exists a complex number $\mu$ such that $F_1 = \mu F_2$, and in particular $\pm \mu=\mu F_2 (1) = F_1 (1) = 1$. If $F_2(1) =1$, then $F_2 =F_1$, and hence property $(3)$ above implies that $F=F_1$ is $\mathbb{C}$-linear. If $F_2(1) =- 1$ then $F_2 =- F_1$, and then property $(3)$ implies that $F=\overline{F_1}$ is conjugate linear.\smallskip

If $\ker (F_2)\nsubseteq \ker (F_1)$, then we can thus find an element $x_0$ in $A$ such that $F_1 (x_0) = 1$ and $F_2 (x_0 ) =0$.\smallskip

Let $h: \mathbb{C}\to \mathbb{C}$ denote the entire function given by $h(\lambda) = e^{i \frac{\pi}{2} \lambda}-1$. Having in mind that $F_1$ and $G_2$ are multiplicative, we deduce from the spectral mapping theorem and $(3)$ that $$ F(e^{i \frac{\pi}{2} x_0})- 1=F(e^{i \frac{\pi}{2} x_0}-1)=F(h(x_0)) $$ $$ = \Re\hbox{e} (F(h(x_0))) + i  \Im\hbox{m} (F(h(x_0)))= \Re\hbox{e} (F_1(h(x_0))) +i \Im\hbox{m} (F_2(h(x_0))) $$
$$=\Re\hbox{e} (F_1(h(x_0))) +i \Im\hbox{m} (F_2(1) G_2(h(x_0))) $$ $$= \Re\hbox{e} (h(F_1(x_0))) +i \Im\hbox{m} (F_2(1) h(G_2(x_0))) $$
  $$= \Re\hbox{e} (h(1)) +i \Im\hbox{m} (F_2(1) h(0)) = -1,$$ and consequently $F(e^{i \frac{\pi}{2} x_0})=0$.
On the other hand, it follows from our hypothesis and the spectral mapping theorem that $$0=F(e^{i \frac{\pi}{2} x_0})\in \mathbb{T} \ \sigma(e^{i \frac{\pi}{2} x_0}) = \mathbb{T} \ e^{i \frac{\pi}{2}\sigma( x_0)} \subseteq \mathbb{T} \ e^{\mathbb{C}},$$ which is impossible.\smallskip

In the general case, the hypothesis imply that $F(1) \in \mathbb{T} \ \sigma(1) = \mathbb{T}$. The mapping $G= \overline{F(1)} F : A\to \mathbb{C}$ is real linear, unital, and satisfies $$G(x) = \overline{F(1)} F(x) \in \overline{F(1)}\ \mathbb{T} \ \sigma (x) = \mathbb{T} \ \sigma (x).$$ It follows from the previous case that $G$ (and hence $F$) is complex linear or conjugate linear.
\end{proof}

Lemma \ref{l KS sphere} provides a tool to replace \cite[Lemma 2.1]{KoSlod} in the Kowalski-S{\l}odkowski theorem.\smallskip

Before dealing with the proof of Proposition \ref{p KS sphere} we introduce some terminology. Let $\displaystyle \mathcal{Q}= \prod_{j=1}^{\infty} [-2^{-j}, 2^{j}]\subseteq \ell_1$ denote the Hilbert cube equipped with the metric given by $\ell_1$. On each interval $[-2^{-j}, 2^{j}]$ we consider the  normalized Lebesgue. Let $\mu$ denote the natural product measure on $\mathcal{Q}$. Following \cite{Mack,KoSlod}, we shall say that a subset $Z$ of a separable Banach space is a \emph{zero set} if for every affine continuous mapping $j :\mathcal{Q} \to X$ with linearly dense image we have $\mu (j^{-1} (Z)) =0$.

\begin{proof}[Proof of Proposition \ref{p KS sphere}] Since $\Delta$ is homogeneous, we can easily see that $\Delta(0)= 0$.\smallskip

As in the proof of \cite[Theorem 1.2]{KoSlod}, we shall assume first that $A$ is separable. \smallskip

Under the hypothesis of our proposition, given $x,y$ in $A$, we have $\Delta (x) -\Delta (y) \in \mathbb{T} \sigma (x-y)$ and hence $$|\Delta (x) -\Delta (y) | \leq \| x-y\|,$$ which guarantees that $\Delta$ is a Lipschitz maps. Theorem 2.3 in \cite{KoSlod} (see also \cite[Theorem 4.4]{Manki73}) assures that $\Delta$ admits real derivatives except for some zero set.\smallskip

Suppose that $\Delta$ admits a real differential at a point $a$. The differential $(D\Delta)_a : A\to \mathbb{C}$ is a real linear mapping defined by $$ (D\Delta)_a (x) =\lim_{r\to 0} \frac{\Delta (a+r x) -\Delta (a)}{r}.$$ By assumptions, the terms in the quotient  $\frac{\Delta (a+r x) -\Delta (a)}{r}\in \mathbb{T} \ \frac{\sigma (r x)}{r} = \mathbb{T} \ \sigma (x),$ and thus $$(D\Delta)_a (x)\in \mathbb{T} \ \sigma (x).$$ Lemma \ref{l KS sphere} proves that $(D\Delta)_a $ is complex linear or conjugate linear.\smallskip

Let $a$ be an element in $A$ such that $\Delta$ admits a real differential at $a$ and $\Delta(a)\neq 0$. In this case, by the homogeneity of $\Delta$ we have $$ (D\Delta)_a (i a) = \lim_{r\to 0} \frac{\Delta (a + i r a) -\Delta(a)}{r} $$ $$= \lim_{r\to 0} \frac{(1+ir) \Delta (a ) -\Delta(a)}{r} = i \Delta(a),$$ and $$(D\Delta)_a ( a ) = \lim_{r\to 0} \frac{\Delta (a + r a) -\Delta(a)}{r} = \Delta(a)\neq 0,$$ and consequently $(D\Delta)_a$ is complex linear. We have proven:  \begin{equation}\label{eq Delta(a) neq 0 and differ implies complex linearity} \Delta(a)\neq 0 \hbox{ and  } \exists (D\Delta)_a \Rightarrow (D\Delta)_a \hbox{ is complex linear}.
\end{equation}

Suppose now that $\Delta(a) =0$ and $\Delta$ admits a real differential at $a$. We shall prove that $(D\Delta)_a$ is complex linear.\smallskip

We first show that $(D\Delta)_a (1)\neq 0$. Since $\Delta$ is homogeneous and $\Delta(a) =0$, we also have $\Delta(\alpha a) =0$ for every $\alpha\in \mathbb{C}$.  Indeed, by assumptions, given $r\in \mathbb{R}\backslash\{0\}$, the element $$\Delta \left(\frac{1}{r} a + 1\right) = \Delta \left(\frac{1}{r} a + 1\right) - \Delta \left(\frac{1}{r} a \right) \in \mathbb{T} \sigma (1) =\mathbb{T},$$ and thus $$(D\Delta)_a (1) = \lim_{r\to 0} \frac{\Delta (a + r 1) -\Delta(a)}{r} = \lim_{r\to 0} \frac{\Delta (a + r 1) }{r} $$ $$= \lim_{r\to 0} \Delta (\frac{1}{r} a + 1) \in \mathbb{T},$$ which proves the desired statement.\smallskip

Let $U:=\{ c\in A : \Delta(c) \neq 0\}$. Since $\Delta$ is a Lipschitz function, the set $U$ is open. It follows from \eqref{eq Delta(a) neq 0 and differ implies complex linearity} that $(D\Delta)_c$ is complex linear for every point $c\in U$ such that $\Delta$ admits a real differential at $c$. Applying \cite[Lemma 2.4]{KoSlod} we conclude that $\Delta$ is holomorphic on $U$, more precisely, for each $c\in U$ and $b\in A$ there exits $\rho = \rho (c,b)>0,$ depending on $c$ and $b$, such that $c +\lambda b\in U,$ for every $|\lambda|<\rho,$ and $f_{c,b} (\lambda) = \Delta (c +\lambda b)$ is holomorphic on $\{\lambda \in \mathbb{C} :|\lambda|<\rho \}$.\smallskip

Having in mind that $\Delta (a) =0$, it follows from the hypothesis that $$\Delta (a +\alpha 1) = \Delta (a +\alpha 1) -\Delta (a) \in \mathbb{T}\ \sigma (\alpha 1) = \alpha \mathbb{T},$$ and thus $\Delta (a +\alpha 1)\neq 0,$ for all $\alpha\in \mathbb{C}\backslash\{0\}$, equivalently, $a +\alpha 1\in U,$  for all $\alpha\in \mathbb{C}\backslash\{0\}$.\smallskip

We consider the continuous function $f :\mathbb{C}\to \mathbb{C},$ $f(\lambda) = \Delta (a + \lambda 1).$ We shall prove that $f$ is holomorphic in $\mathbb{C}\backslash\{0\}$. Namely, fix $\lambda_0\in \mathbb{C}\backslash\{0\}$. Since $c=a +\lambda_0 1\in U$, it follows from the above paragraph with $b =1$ that there exists $\rho>0$ such that $c +\lambda b\in U,$ for every $|\lambda|<\rho,$ and $f_{c,b} (\lambda) = \Delta (c +\lambda b)= \Delta (a +\lambda_0 1 +\lambda 1) = f(\lambda_0 +\lambda)$ is holomorphic on $\{\lambda \in \mathbb{C} :|\lambda|<\rho \}$, witnessing that $f$ is holomorphic at $\lambda_0$.\smallskip

Since $f$ is continuous on $\mathbb{C}$ and  holomorphic in $\mathbb{C}\backslash\{0\}$ it must be an entire function. In particular there exists $\displaystyle f' (0) = \lim_{h\in \mathbb{C}, h\to 0} \frac{f(h)-f(0)}{h}$, and in particular $$(D\Delta)_a (1) =  \lim_{r\to 0} \frac{\Delta (a + r 1) -\Delta(a)}{r} =  \lim_{r\in \mathbb{R}, r\to 0} \frac{f(r)-f(0)}{r} =  f' (0) $$ $$= \lim_{r\in \mathbb{R}, r\to 0} \frac{f(i r)-f(0)}{i r} =  \lim_{r\to 0} \frac{\Delta (a + i r 1) -\Delta(a)}{i r} = \frac{1}{i} (D\Delta)_a (i 1). $$ Since $(D\Delta)_a (1)\neq 0,$ $(D\Delta)_a (i 1) = i (D\Delta)_a (1)$ and $(D\Delta)_a$ must be complex linear or conjugate linear, we conclude that $(D\Delta)_a$ is complex linear.\smallskip

We have therefore show that $(D\Delta)_a$ is complex linear for all point $a\in A$ at which $\Delta$ admits a real differential.\smallskip

We can therefore apply \cite[Lemma 2.4]{KoSlod} (with $U= A$) to conclude that for each $a,b\in A$ the mapping $\varphi: \mathbb{C}\to \mathbb{C}$, $\varphi(\lambda) := \Delta (\lambda a +b)$ is entire with $\varphi (\lambda) -\varphi(\mu) \in \mathbb{T} \ \sigma ((\lambda-\mu) a),$ and thus $| \varphi (\lambda) -\varphi(\mu)| \leq \|a\| \ |\lambda-\mu |$. It is well known from the classical theory of holomorphic functions that $\varphi$ must be affine, and consequently, \begin{equation}\label{eq Delta affine} \Delta (\lambda a +b)=\varphi(\lambda ) = \lambda (\Delta (a +b) -\Delta (b))  + \Delta (b),
 \end{equation} for every $a,b\in A$, $\lambda\in \mathbb{C}$. As in \cite[Proof of Theorem 1.2]{KoSlod}, with $b=0$ we get $\Delta(\lambda a) =\lambda \Delta (a),$ and given $c,d\in A$, by replacing in \eqref{eq Delta affine} $a$, $b$ and $\lambda$ with $\frac{c-d}{2}$, $d$ and $2$, respectively, we have $$2 \Delta \left(\frac{c+d}{2}\right) = \Delta(c) +\Delta(d).$$ So, $\Delta$ is linear and the rest follows from Proposition \ref{p GKZ sphere}.\smallskip

When $A$ is not separable, we can restrict $\Delta$ to the subalgebra generated by any two elements which is always separable, and then the conclusion follows from the arguments above.\smallskip

Finally, since $\Delta : A\to \mathbb{C}$ is linear, $\Delta(0) = 0$ and $\Delta(a) \in \mathbb{T} \sigma(a)$ for every $a$ in $A$, the final statement is a consequence of Proposition \ref{p GKZ sphere}.\end{proof}

We can now deal with 2-local isometries between Lip$(E)$ spaces.

\begin{theorem}\label{t 2local isometries between Lip spaces} Let $E$ and $F$ be metric spaces, and let us assume that the set Iso$((\hbox{Lip}(E),\|.\|_{s}),(\hbox{Lip}(F),\|.\|_{s}))$ is canonical. Then every\hyphenation{every} weak-2-local Iso$((\hbox{Lip}(E),\|.\|_{s}),(\hbox{Lip}(F),\|.\|_{s}))$-map $\Delta$ from $\hbox{Lip}(E)$ to $\hbox{Lip}(F)$ is a linear map. Furthermore, the same conclusion holds when the norm $\|.\|_{s}$ is replaced with the norm $\|.\|_{_L}$.
\end{theorem}

\begin{proof} Let $\Delta : \hbox{Lip}(E)\to \hbox{Lip}(F)$ be a weak-2-local isometry with respect to the norm $\|.\|_s$. It is known that $\Delta$ is 1-homogeneous (i.e., $\Delta(\alpha f)= \alpha \Delta(f)$, for all $\alpha\in \mathbb{C}$ and all $f\in$Lip$(E)$), and $\Delta(0) =0$ (compare \cite[Lemma 2.1]{CaPe2017}). We fix now an element $s\in F$, and we consider the mapping $\Delta_s = \delta_s \circ \Delta :\hbox{Lip}(E) \to \mathbb{C}$. Since, by hypothesis, given $f,g\in $Lip$(E)$, there exist $\tau_{f,g,s}\in \mathbb{T}$ and a surjective isometry $\varphi_{f,g,s} : F\to E$ such that $$\delta_s \Delta (f) = \delta_s(\tau_{f,g,s} f(\varphi_{f,g,s}(.))), \hbox{ and } \delta_s \Delta (g) = \delta_s(\tau_{f,g,s} g(\varphi_{f,g,s}(.))),$$ and then $$\Delta_s(f)-\Delta_s(g) = \tau_{f,g,s} (f(\varphi_{f,g,s}(s)) -g(\varphi_{f,g,s}(s))) \in \mathbb{T} \ \sigma(f-g).$$ Since $(\hbox{Lip}(E),\|.\|_{s})$ is a unital complex Banach algebra, we are thus in conditions to apply Proposition \ref{p KS sphere} to conclude that $\Delta_s$ is a linear map. The linearity of $\Delta$ follows from the arbitrariness of $s$.\smallskip

For the last statement, the space $(\hbox{Lip}(E),\|.\|_{_L})$ is not formally a complex Banach algebra. However, by the arguments given above, for each weak-2-local isometry $\Delta : (\hbox{Lip}(E),\|.\|_{_L})\to (\hbox{Lip}(E),\|.\|_{_L})$ and each $s\in F$ we have $\Delta(0)=0$ and $$\delta_s \circ \Delta(f)-\delta_s \circ \Delta (g)\in  \mathbb{T} \ (f-g)(E) =  \mathbb{T} \ \sigma_{_{({Lip}(E),\|.\|_{s})}}(f-g),$$ for all $f,g\in \textrm{Lip}(E)$. Since the spaces and algebras underlying \linebreak $(\hbox{Lip}(E),\|.\|_{_L})$ and  $(\hbox{Lip}(E),\|.\|_{s})$ coincide, we apply the conclusion in the first paragraph to conclude the proof.
\end{proof}

\begin{remark}\label{remark 2-local weaker assumptions than canonical} If in Theorem \ref{t 2local isometries between Lip spaces} the hypothesis $$\hbox{Iso$((\hbox{Lip}(E),\|.\|),(\hbox{Lip}(F),\|.\|))$ being canonical}$$ (where $\|.\|$ stands for $\|.\|_s$ or for $\|.\|_{_L}$), is replaced by the weaker assumption that every element in Iso$((\hbox{Lip}(E),\|.\|),(\hbox{Lip}(F),\|.\|))$ is of the form $$T(f ) (s)  = \tau (s) \ f(\varphi(s)), \ \hbox{ for all } f\in \hbox{Lip}(E), s\in F,$$ where $\tau$ is a unimodular function in Lip$(F)$ and $\varphi: F \to E$ is a surjective isometry, then the conclusion of Theorem \ref{t 2local isometries between Lip spaces} remains valid.
\end{remark}

It should be noted here that, under the hypothesis of the previous theorem (i.e., the set Iso$((\hbox{Lip}(E),\|.\|_{_L}),(\hbox{Lip}(F),\|.\|_{_L}))$ is canonical), we can also prove a variant of \cite[Theorem 2.1]{JV11} in the setting of 2-local isometries from $\hbox{Lip}(E)$ to $\hbox{Lip}(F)$.

\begin{corollary}\label{cor JV thm 2.1 norm L} Let $E$ and $F$ be compact metric spaces, and let us assume that the set Iso$((\hbox{Lip}(E),\|.\|_{_L}),(\hbox{Lip}(F),\|.\|_{_L}))$ is canonical. Then every\hyphenation{every} 2-local Iso$((\hbox{Lip}(E),\|.\|_{_L}),(\hbox{Lip}(F),\|.\|_{_L}))$-map $\Delta$ from $\hbox{Lip}(E)$ to $\hbox{Lip}(F)$ is a linear isometric map and there exist a closed subset $F_0\subset F$, a Lipschitz map $\varphi: F_0\to X$ with $L(\varphi)\leq \max\{1, \hbox{diam}(X)\}$ and $\tau\in \mathbb{T}$ such that $$\Delta(f) (s) = \tau \ f(\varphi(s)), \hbox{ for all } f\in \hbox{Lip} (X), s\in F_0.$$
\end{corollary}

\begin{proof} Let $\Delta:\hbox{Lip}(E) \to \hbox{Lip}(F)$ be a 2-local isometry. Since $\Delta$ is a weak-2-local isometry, we deduce from Theorem \ref{t 2local isometries between Lip spaces} that $\Delta$ is linear. $\Delta$ being a 2-local isometry implies that $\Delta$ is a linear isometry. The hypotheses also show that $\Delta(1_{_E}) = \tau_{1_{_E}}$ is a constant unimodular function. The desired conclusion follows from Theorem 2.4 in \cite{JV2008} and the facts that Iso$((\hbox{Lip}(E),\|.\|_{_L}),(\hbox{Lip}(F),\|.\|_{_L}))$ is canonical, and $\Delta$ is a local isometry.
\end{proof}

We can also obtain some other interesting consequences derived from Theorem \ref{t 2local isometries between Lip spaces}. The following corollaries complement the conclusions in \cite{JV11}.

\begin{corollary}\label{c 2-local isometries on LipX compact sum norm} Let $K$ be a compact metric space. Suppose that the group Iso$(\hbox{Lip}(K), \|.\|_{_s})$ is canonical. Then every 2-local isometry $\Delta : (\hbox{Lip}(K), \|.\|_{_s})\to (\hbox{Lip}(K), \|.\|_{_s})$ is a surjective isometry.
\end{corollary}

\begin{proof} By Theorem \ref{t 2local isometries between Lip spaces} every 2-local isometry $\Delta : (\hbox{Lip}(K), \|.\|_{_s})\to (\hbox{Lip}(K), \|.\|_{_s})$ is a linear local isometry, and thus $\Delta$ is a local linear isometry. Now, applying Theorem 2.3 in \cite{JiMorVill2010} we derive that $\Delta$ is a surjective linear isometry.
\end{proof}

\begin{corollary}\label{c 2-local isometries on LipX compact max norm} Let $K$ be a compact metric space. Suppose $K$ is connected with diameter at most 1 {\rm(}or satisfies certain separation property to guarantee that Iso$(\hbox{Lip}(K), \|.\|_{_L})$ is canonical{\rm)}. Then every 2-local isometry $\Delta : (\hbox{Lip}(K), \|.\|_{_L})\to (\hbox{Lip}(K), \|.\|_{_L})$ is a surjective isometry.
\end{corollary}

\begin{proof} Let $\Delta : (\hbox{Lip}(K), \|.\|_{_L})\to (\hbox{Lip}(K), \|.\|_{_L})$ be a 2-local isometry. Theorem \ref{t 2local isometries between Lip spaces} implies that $\Delta$ is linear. Therefore, $\Delta$ is a linear 2-local isometry.\smallskip

Although Theorem 2.3 in \cite{JiMorVill2010} is only stated for the norm $\|.\|_s$, the rest of our arguments owe too much to the original proof by Jim{\'e}nez-Vargas, Morales Campoy, and Villegas-Vallecillos in \cite{JiMorVill2010}, we include a brief argument for completeness reasons.\smallskip

Since $\Delta$ is a linear 2-local isometry and Iso$(\hbox{Lip}(K), \|.\|_{_L})$ is canonical, we conclude that $\Delta (1) =\tau_1\in \mathbb{T}$ is a constant function.\smallskip

By Theorem \ref{t weak local isometries between Lip spaces} there exist a unimodular $\tau\in \hbox{Lip}(K)$ and an algebra homomorphism $\psi : \hbox{Lip}(K)\to \hbox{Lip}(K)$ such that $$\Delta (f) = \tau \ \psi (f), \ \ \hbox{ for all } f\in \hbox{Lip}(K).$$ For each $s\in K$, the mapping $\delta_s \circ \psi: \hbox{Lip} (K)\to \mathbb{C}$ is a non-zero multiplicative functional. Otherwise, $\delta_s \circ \psi (f) = 0= \delta_s \circ \Delta (f)$ for all $f\in \hbox{Lip}(K)$. If we pick a nowhere vanishing function $f_0\in \hbox{Lip}(K)$, then by hypothesis there exist $\tau_{f_0}\in \mathbb{T}$ and a surjective isometry $\varphi_{f_0} : K\to K$ such that $0=\Delta(f_0) (s) = \tau_{f_0} \ h(\varphi_{f_0} (s)),\ \forall s\in K,$ which impossible.\smallskip

Since, non-zero multiplicative linear functionals on $(\hbox{Lip}(K), \|.\|_{_L})$ and on $(\hbox{Lip}(K), \|.\|_{s})$ are evaluation maps at a (unique) point in $K$ (see \cite[Theorem 4.3.6]{Weav1999} or \cite[page 199]{JiMorVill2010}), then there exits a unique $\varphi(s)\in K$ such that $\delta_s \circ \psi =\delta_{\varphi(s)}$. We have thus defined a function $\varphi: K\to K$ satisfying $$\delta_s \circ \psi (f) =\delta_{\varphi(s)} (f),\ \hbox{ for all } f\in \hbox{Lip}(E).$$

Following the ideas in \cite[page 199]{JiMorVill2010} we prove that $\varphi$ is injective. Namely, suppose that $\varphi (x) = \varphi (y)$, and let us take $h\in \hbox{Lip}(K)$ such that $h^{-1} (\{0\}) = \{\varphi(x)\}$. By the assumptions, there exist $\tau_{h}\in \mathbb{T}$ and a surjective isometry $\varphi_h : K\to K$ such that $$\Delta(h) (s) = \tau_h \ h(\varphi_h (s)),\ \forall s\in K.$$ Consequently, $$ \tau_h \ h(\varphi_h (x)) = \Delta(h) (x) = \tau(x) \ \psi(h) (x) =  \tau(x) \ h (\varphi (x)) = 0,$$ and $$ \tau_h \ h(\varphi_h (y)) = \Delta(h) (y) = \tau(y) \ \psi(h) (y)  =  \tau(y) \ h (\varphi (y)) = 0,$$ which guarantees that $\varphi_h(x) = \varphi(x) = \varphi_h(y)$, and thus $x=y$.\smallskip

We shall next check, following \cite[proof of Theorem 2.3]{JiMorVill2010}, that $\varphi$ is an isometry. Let us take $x\neq y$ in $K$, and $k\in \hbox{Lip} (K)$ defined by $\displaystyle k(z):= \frac{d(z,\varphi(x))}{d(z,\varphi(x))+d(z,\varphi(y))}$. Clearly, $k^{-1}(\{0\})= \{\varphi(x)\}$ and $k^{-1}(\{1\})= \{\varphi(y)\}$. By hypothesis, there exist $\tau_{k}\in \mathbb{T}$ and a surjective isometry $\varphi_k : K\to K$ such that $$ \tau_k \ k(\varphi_k (x)) = \Delta(k) (x) = \tau(x) \ \psi(k) (x) =  \tau(x) \ k (\varphi (x)) = 0,$$ and $$ \tau_k \ k(\varphi_k (y)) = \Delta(k) (y) = \tau(y) \ \psi(k) (y)  =  \tau(y) \ k (\varphi (y)) = \tau(y).$$ We deduce from the properties of $k$ that $\varphi_k (x)=\varphi (x)$ and $\varphi_k (y) = \varphi (y)$. Since $\varphi_k$ is an isometry we get $d(\varphi (x),\varphi (x)) = d(\varphi_k (x),\varphi_k (y)) = d(x,y)$ as desired.\smallskip

Finally, since $\varphi: K \to K$ is an isometry on a compact metric space, Lemma 2.1 in \cite{JiMorVill2010} implies that $\varphi$ is surjective. Therefore the identity $$\Delta (f) (s) = \tau (s) \ \psi (f) (s) = \tau (s) \ f(\varphi(s)),$$ holds for all $f\in \hbox{Lip}(K)$ and all $s\in K$. In particular, $\mathbb{T} \ni \tau_1=\Delta(1) = \tau$ is a constant function.
\end{proof}

Back to the setting of uniform algebras, we can now combine our spherical variant of the Kowalski-S{\l}odkowski theorem (see Proposition \ref{p KS sphere}) with the deLeeuw-Rudin-Wermer theorem (see \cite[Corollary 2.3.16]{FleJa03}) to study weak-2-local isometries.

\begin{theorem}\label{t 2-local uniform algebras} Let $A$ be a uniform algebra, let $Q$ be a compact Hausdorff space, and suppose that $B$ is a norm closed subalgebra of $C(Q)$ containing the constant functions. Then every weak-2-local isometry {\rm(}respectively, every weak-2-local {\rm(}algebraic{\rm)} isomorphism{\rm)} $\Delta : A\to B$ is a linear map.
\end{theorem}

\begin{proof} Let $\Delta : A\to B$ be a weak-2-local isometry. We have already commented that $\Delta$ is homogeneous (see \cite[Lemma 2.1]{CaPe2017}). If we fix an arbitrary $s\in Q$, the mapping $\delta_s \circ \Delta : A\to \mathbb{C}$ satisfies the hypothesis of Proposition \ref{p KS sphere}. Indeed, since $\Delta$ is a weak-2-local isometry, given $a,b\in A$, by the deLeeuw-Rudin-Wermer theorem \cite[Corollary 2.3.16]{FleJa03}, there exists an algebra isomorphism $\pi_{a,b,s} : A\to B$ and a unimodular $h_{a,b,s}\in B$ such that $$\delta_s \circ \Delta(a) = h_{a,b,s}(s) \ \pi_{a,b,s} (a) (s), \hbox{ and } \delta_s \circ \Delta(b) = h_{a,b,s} (s) \ \pi_{a,b,s}(b) (s).$$ Therefore $$ \delta_s \circ \Delta(a) - \delta_s \circ \Delta(b) = h_{a,b,s} (s) \ \pi_{a,b,s}(a-b) (s) \in \mathbb{T} \ \sigma (a-b),$$ as desired. Proposition \ref{p KS sphere} assures that $ \delta_s \circ \Delta$ is linear. Therefore $\Delta$ is linear by the arbitrariness of $s$.\smallskip

The statement concerning isomorphisms is a clear consequence of the fact that every isomorphism between $A$ and $B$ is an isometry.
\end{proof}

Since 2-local isomorphisms and 2-local isometries between uniform algebras are a weak-2-local isometries, Theorem \ref{t 2-local uniform algebras} provides a positive answer to Problems \ref{problem linearity of 2-local isometries for uniform} and \ref{problem linearity of n-local isometries for uniform}.

\begin{corollary}\label{c 2-local-isometries on uniform algebras} Let $A$ and $B$ be uniform algebras. Then every 2-local isometry {\rm(}respectively, every 2-local {\rm(}algebraic{\rm)} isomorphism{\rm)} $\Delta : A\to B$ is a linear map.\smallskip

Furthermore, if every local isometry {\rm(}respectively, every 2-local {\rm(}algebraic{\rm)} isomorphism{\rm)} from $A$ into $B$ is a surjective isometry {\rm(}respectively, an isomorphism{\rm)}, then every 2-local isometry {\rm(}respectively, every 2-local isomorphism{\rm)} $\Delta : A\to B$ is a surjective linear isometry  {\rm(}respectively, an isomorphism{\rm)}.$\hfill\Box$
\end{corollary}

\smallskip
\noindent\textbf{Acknowledgements:} L. Li was partly supported by NSF of China project no. 11301285. A.M. Peralta was partially supported by the Spanish Ministry of Economy and Competitiveness and European Regional Development Fund project no. MTM2014-58984-P and Junta de Andaluc\'{\i}a grant FQM375. L. Wang was partly supported by NSF of China Grants No. 11371222 and 11671133. Y.-S. Wang was partly supported by Taiwan MOST 104-2115-M-005-001-MY2.\smallskip

Most of the results presented in this note were obtained during a visit of A.M. Peralta at the School of Mathematical Sciences in Nankai University. He would like to thank the first author and the Department of Mathematics for the hospitality during his stay.

\end{document}